\documentclass[draft]{amsart}

\usepackage{amssymb}
\usepackage{amstext}
\usepackage{amsmath}
\usepackage{amsthm}
\usepackage{latexsym}
\usepackage{amsfonts}
\usepackage{graphicx}
\usepackage{enumerate}

\newtheorem{thm}{Theorem}[section]
\newtheorem{cor}[thm]{Corollary}
\newtheorem{prop}[thm]{Proposition}
\newtheorem{lem}[thm]{Lemma}

\theoremstyle{remark}
\newtheorem{rem}[thm]{Remark}
\newtheorem*{pf}{{\sl Proof}}

\newtheorem*{pfp}{{\sl Proof of Theorem \ref{abthm}}.(1)}
\newtheorem*{pfd}{{\sl Proof of Theorem \ref{abthm}}.(2)}
\newtheorem*{pfa}{{\sl Proof of Theorem \ref{ARC}}}

\theoremstyle{definition}
\newtheorem{dfn}[thm]{Definition}
\newtheorem{ex}[thm]{Example}

\numberwithin{equation}{section}
\numberwithin{thm}{section}

\def\Hom{\operatorname{Hom}}
\def\Ext{\operatorname{Ext}}
\def\mod{\operatorname{mod}}

\def\Coker{\operatorname{Coker}}
\def\m{\mathfrak m}
\def\P{\mathbb P}
\def\AB{\mathcal{AB}}
\def\p{\operatorname{P}}
\def\s{{}^{\ast}}
\def\depth{\operatorname{depth}}
\def\pdim{\operatorname{\P -dim}}
\def\pd{\operatorname{pd}}

\def\cidim{\operatorname{CI-dim}}
\def\abdim{\operatorname{AB-dim}}
\def\gdim{\operatorname{G-dim}}
\def\cmdim{\operatorname{CM-dim}}

\begin{document}






\title{A Homological dimension related to AB rings}


%

\author{Tokuji Araya}
\address{Liberal Arts Division, Tokuyama College of Technology, Gakuendai, Shunan, Yamaguchi, 745-8585, Japan}

\email{araya@tokuyama.ac.jp}


%

\thanks{The author was supported by JSPS Research Activity Start-up 23840043.}

%

\begin{abstract}
There are many homological dimensions which are closely related to ring theoretic properties.
The notion of a AB ring has been introduced by Huneke and Jorgensen.
It has nice homological properties.
In this paper, we shall define a homological dimension which is closely related to a AB ring, and investigate its properties.
\end{abstract}


\maketitle



\section{Introduction}

Throughout this paper, let $R$ be a commutative noetherian local ring with maximal ideal $\m$ and residue field $k$.
We denote by $\mod R$ the category of finitely generated $R$-modules.
All modules considered in this paper are assumed to be finitely generated.

In commutative ring theory, there are important ring theoretic properties $\P$ with implications below:

\begin{center}
$\P$ : regular $\Rightarrow$ complete intersection $\Rightarrow$ Gorenstein $\Rightarrow$ Cohen-Macaulay.
\end{center}

Related to these properties, there are homological dimensions $\pdim_R M$ for $R$-modules $M$ with following:

\begin{center}
$
\pdim_R M : \pd_R M \ge \cidim_R M \ge \gdim_R M \ge \cmdim_R M.
$
\end{center}

The relationship between a ring theoretic property $\P$ and a related homological dimension $\pdim_R$ is the following (c.f. \cite{AB1,AB2,S} for the regular property, \cite{AGP} for the complete intersection property, \cite{ABr} for the Gorenstein property and \cite{G} for the Cohen-Macaulay property, and see also \cite{Av}):

\begin{prop}
\begin{enumerate}[\rm (1)]
\item The following conditions are equivalent.
\begin{enumerate}[\rm (i)]
\item $R$ satisfies property $\P$.
\item $\pdim_R M < \infty$ for every $R$-module $M$.
\item $\pdim_R k < \infty$.
\end{enumerate}
\item Let $M$ be a nonzero $R$-module. If $\pdim_R M < \infty$, then $\pdim_R M = \depth R - \depth M$. 
\end{enumerate}
\end{prop}

We say that an $R$-module $M$ satisfies {\it Auslander condition (AC)} if there exists an integer $d$ such that $\Ext^{>d}_R(M,N)=0$ for all $R$-modules $N$ with $\Ext^{\gg 0}_R(M,N)=0$.
We say that $R$ is an {\it AC ring} if every $R$-modules satisfies (AC).
The notion of a AB ring has been introduced by Huneke and Jorgensen \cite{HJ}.
They define an {\it AB ring} as a Gorenstein AC ring.
They proved that every complete intersection is AB.

In this paper, we shall define AB-dimension which satisfies the following properties.

\begin{thm}\label{abthm}
\begin{enumerate}[\rm (1)]
\item The following conditions are equivalent.
\begin{enumerate}[\rm (i)]
\item $R$ is an AB ring.
\item $\abdim_R M < \infty$ for every $R$-module $M$.
\item $\abdim_R L < \infty$ for every $R$-module $L$ of finite length.
\item $R$ is Gorenstein and the class $\AB$ of $\mod R$ consisting of all modules of finite AB-dimension is closed under extension.
\end{enumerate}
\item Let $M$ be a nonzero $R$-module.
\begin{enumerate}[\rm (i)]
\item If $\abdim_R M < \infty$, then $\abdim_R M=\depth R-\depth M$.
\item $\cidim_R M \ge \abdim_R M \ge \gdim_R M$.
\end{enumerate}
\end{enumerate}
\end{thm}

In representation theory of finite dimensional algebras, there are many homological conjectures.
The Auslander condition is closely related to these conjectures (see \cite{H}).
The Auslander-Reiten conjecture (ARC) below is one of these conjectures.

\medskip

\noindent
(ARC) \quad
Let $R$ be a (not necessarily commutative) noetherian ring and $M$ be a finitely generated $R$-module.
If $\Ext^{>0}_R(M,M\oplus R) = 0$, then $M$ is projective.

\medskip

Christensen and Holm \cite{CH} proved that a (not necessarily commutative) AC ring satisfies (ARC) .
In the end of this paper, we will prove the following theorem.

\begin{thm}\label{ARC}
Let $M$ be an $R$-module which has finite AB-dimension.
If $\Ext^{>0}_R(M,M)=0$,
then $M$ is free.
\end{thm}

\section{AB-dimension}

In this section, we shall define AB-dimension and investigate some properties.

Let $M$ be an $R$-module.
Let
$$
\cdots \to F_n \overset{\partial _n}{\to} F_{n-1} \to \cdots \to F_1 \overset{\partial _1}{\to} F_0 \to M \to 0
$$
be a minimal free resolution of $M$.
The $n$th {\em syzygy} module $\Omega^nM$ of $M$ is defined as the image of the map $\partial_n$.
We denote the first syzygy module $\Omega M$ for simple.
Note that the $n$th syzygy module is uniquely determined up to isomorphism.

We recall the definition of G-dimension.

\begin{dfn}\cite{ABr}
Let $M$ be an $R$-module.
$M$ is called {\it totally reflexive} if $M\cong M\s\s$ and if $\Ext^{>0}_R(M,R)=\Ext^{>0}_R(M\s,R)=0$.
Here, $(-)\s=\Hom_R(-,R)$.
We say that the G-dimension of $M$ is at most $n$ if the $n$th syzygy module $\Omega ^n M$ is totally reflexive.
\end{dfn}

We give some notations to define an AB-dimension.

\begin{dfn}\label{defAB}
Let $M$ be an $R$-modules.
\begin{enumerate}[\rm (1)]
\item For an $R$-module $N$, we set $\p _R(M,N) := \sup \{ \ n\ |\ \Ext^n_R(M,N) \not= 0\ \}$.
\item We denote by $M^{\perp}$ the full sub category of $\mod R$ consisting of all $R$-modules $N$ with $\Ext^{\gg 0}_R(M,N)=0$.
\item We set $\p _R(M) := \sup \{\ \p _R(M,N)\ |\ N \in M^{\perp}\ \}$.
\item We define $\abdim _RM := \sup \{\ \gdim _RM,\p _R(M)\ \}$.
\end{enumerate}
\end{dfn}

There are some remarks.

\begin{rem}\label{rem}
\begin{enumerate}[\rm (1)]
\item If $\pd _RM<\infty$, then $M^{\perp}=\mod R$.
For any nonzero $R$-module $N$, one can check that $\p _R(M,N)=\pd_R M=\depth R - \depth M$ by Nakayama's lemma (c.f. \cite[Theorem 4.2]{AY}).
Therefore we have $\p _R(M)=\depth R - \depth M$.
\item Let $0 \to M_1 \to M_2 \to M_3 \to0$ be an exact sequence and let $(i,j,l)$ be a permutation of $(1,2,3)$.
Then one can easily check $M_i^{\perp} \cap M_j^{\perp} \subset M_l^{\perp}$.
In particular, if $\pd_RM_i<\infty$, then $M_j^{\perp}=M_l^{\perp}$.
\item If either $M$ or $N$ is zero, then $\{ \ n\ |\ \Ext^n_R(M,N) \not= 0\ \}$ is empty.
In this case, we define $\p _R(M,N)=-\infty$.
In particular, $\p _R(0)$ is $-\infty$.
If $M^{\perp}=\{ 0\}$, then $\p _R(M)$ is also $-\infty$.
\item If there exists a nonzero $R$-module $N$ in $k^{\perp}$, then $N$ has finite injective dimension and therefore $R$ must be Cohen-Macaulay by Bass' conjecture \cite{PS, Ho, R}.
Thus if $R$ is not Cohen-Macaulay, then we have $\p _R(k)=-\infty$.
On the other hand, if $R$ is Cohen-Macaulay, then there exists nonzero $R$-module $I$ of finite injective dimension.
Since $\p _R(M,I)=\depth R - \depth M$ for any nonzero $R$-module $M$ (c.f. \cite[Exercises 3.1.24]{BH}), we have $\p _R(M) \geq \depth R - \depth M \geq 0$ and $\p _R(k)=\depth R$.
In both case, we see that the residue field $k$ always satisfies (AC).
Therefore we have $\abdim_Rk<\infty$ if and only if $R$ is Gorenstein.
\end{enumerate}
\end{rem}

Let $M$ be an $R$-module.
If $\gdim_R M < \infty$, then there exist following exact sequences which are called {\it Cohen-Macaulay approximation} and {\it finite projective hull} respectively (c.f. \cite[Theorem 1.1]{ABu}):

$$
0 \to Y_M \to X_M \to M \to 0,
$$
$$
0 \to M \to Y^M \to X^M \to 0,
$$
where $Y_M$ and $Y^M$ have finite projective dimensions, and $X_M$ and $X^M$ are totally reflexive.

\begin{lem}\label{ABfin}
Let $0 \to M_1 \to M_2 \to M_3 \to $ be an exact sequence and let $(i,j,l)$ be a permutation of $(1,2,3)$.
Assume $\pd _RM_i<\infty$.
Then $\p _R(M_j)<\infty$ if and only if $\p _R(M_l)<\infty$.
In particular, let $0 \to Y_M \to X_M \to M \to 0$ and $0 \to M \to Y^M \to X^M \to 0$ be a Cohen-Macaulay approximation and a finite projective hull respectively respectively, then the finiteness of AB-dimension of $X_M$, AB-dimension of $M$ and AB-dimension of $X^M$ are coincide.
\end{lem}

\begin{pf}
Assume $\p _R(M_j)<\infty$.
Note that $M_j^{\perp}=M_l^{\perp}$ by Remark \ref{rem}.(2).
Let $N$ be a nonzero $R$-module in $M_l^{\perp}$.
Note that $\p _R(M_i,N)=\pd_R M_i < \infty$  by Remark \ref{rem}.(1).
Applying $\Hom_R (-,N)$ to $0\to M_1 \to M_2 \to M_3 \to 0$, 
we have $\p _R(M_l,N) \leq \max \{\ \p _R(M_i,N)+1, \p _R(M_j,N)+1\ \} \leq \max \{\ \pd_R M_i+1, \p _R(M_j)+1\ \}$.
This yields $\p _R(M_l) \leq \max \{\ \pd_R M_i+1, \p _R(M_j)+1\ \}<\infty$.
\qed
\end{pf}

The following gives a more strong statement than Theorem \ref{abthm}.(2).(i).

\begin{lem}\label{keylem}
Let $M$ and $N$ be nonzero $R$-modules.
Assume $\abdim_RM<\infty$.
If $\p_R(M,N)<\infty$,
then $\p_R(M,N)=\depth R-\depth M$.
\end{lem}

\begin{pf}
Put $t=\depth R-\depth M$.
Since $\gdim _RM=t$, we have $\p_R(M,R)=t$.
If $t<\p_R(M,N)$, then we have $\p_R(M,\Omega^nN)=\p_R(M,N)+n < \infty$ for every $n \ge 0$.
This contradicts to $\p_R(M)<\infty$.
Thus $\p_R(M,N) \leq t$.

If $t=0$, then $0 \leq \p_R(M,N)\leq t=0$.
Assume $t>0$.
Let $0 \to M \to Y^M \to X^M \to 0$ be a finite projective hull of $M$.
By depth lemma, we see $\depth Y^M=\depth M$.
Since $\pd _RY^M<\infty$, we have $\p_R(Y^M,N)=\pd _RY^M=t<\infty$.
It follows from Lemma \ref{ABfin}, $\abdim_R X^M<\infty$.
Since $\depth X^M=\depth R$, we have $\p_R(X^M,N)=0$.
Applying $\Hom_R(-,N)$ to $0 \to M \to Y^M \to X^M \to 0$, we get $\p_R(M,N)=t$.
\qed
\end{pf}

Now we can prove Theorem \ref{abthm}.(2).

\begin{pfd}
(i) is clear by Lemma \ref{keylem}.
To prove (ii), we assume $\cidim _RM<\infty$.
We remark that it is enough to show $\abdim _RM<\infty$ by (i).
It comes from \cite[Theorem (1.4)]{AGP}, we have $\gdim _RM< \infty$.
For any nonzero $R$-module $N$ with $\p_R(M,N)<\infty$, we have $\p_R(M,N)=\depth R-\depth M$ by \cite[Theorem 4.2]{AY}.
This yields $\p_R(M)=\depth R-\depth M<\infty$ and we see $\abdim _RM <\infty$.
\qed
\end{pfd}

We give some fundamental properties of AB-dimension.

\begin{prop}\label{prop}
Let $M$ be a nonzero $R$-module.
The followings hold:
\begin{enumerate}[\rm (1)]
\item $\abdim_R\Omega M=\sup \{ \abdim _RM-1,0\}$.
\item $\abdim _RM/xM=\abdim _RM+1$ for any $M$-regular element $x \in \m$.
\item $\abdim _{R/xR}M/xM=\abdim _RM$ for any $M$- and $R$-regular element $x \in \m$.
\end{enumerate}
\end{prop}

\begin{proof}
We remark that Auslander and Bridger \cite{ABr} proved each equalities for G-dimension.
Thus we may assume $\gdim_RM<\infty$.
In this situation, it is enough to show that we only check the Auslander condition (AC).

(1) and (2) are Trivial. Indeed, for any $R$-module $N$, we can check that $\p_R(\Omega M,N)=\sup \{ \p_R(M,N)-1,0\}$ and $\p_R(M/xM,N)=\p_R(M,N)+1$ respectively.

To show (3), we assume $\p_R(M)<\infty$.
Note that $\p_R(M/xM)<\infty$ by (2).
Let $N$ be an $R/xR$-module with $\p_{R/xR}(M/xM,N)<\infty$.
Then $\p_{R/xR}(M/xM,N)=\p_R(M/xM,N)-1$ (see \cite[Lemma 2.6]{AY}).
This yields $\p_{R/xR}(M/xM)=\p_R(M/xM)-1<\infty$.

Conversely, assume $\p_{R/xR}(M/xM)<\infty$.
Let $N$ be an $R$-module with $\p_{R}(M,N)<\infty$.

If $x$ is an $N$-regular element, then we can see $\p_R(M,N)=\p_R(M,N/xN)$ by Nakayama's lemma.
Since $x$ is an $M$-regular element, there are isomorphisms $\Ext^i_R(M,N/xN)\cong \Ext^i_{R/xR}(M/xM,N/xN)$ for all $i \ge 0$.
Thus we have equalities $\p_R(M,N)=\p_R(M,N/xN)=\p_{R/xR}(M/xM,N/xN)=\p_{R/xR}(M/xM)=\depth R/xR-\depth M/xM=\depth R-\depth M$ by Lemma \ref{keylem}.

If $x$ is not an $N$-regular element, then $N$ is not free.
Note that $x$ is an $\Omega N$-regular element.
It comes from above argument, $\p_R(M,\Omega N)=\depth R-\depth M$.
Since $\p_R(M,R)=\depth R-\depth M$, we have $\p_R(M,N)\le \depth R-\depth M$.

Hence we get $\p_R(M)<\infty$.
\end{proof}

Now we can prove the Theorem \ref{abthm}.(1).

\begin{pfp}
(i) $\iff$ (ii): Obvious by definition.

(i)(,(ii)) $\Rightarrow$ (iv) is trivial.

(iv) $\Rightarrow$ (iii): Since $R$ is Gorenstein, we have $k \in \AB$ by Remark \ref{rem}.(4).
Therefore $\AB$ contains all $R$-modules of finite length by the assumption.

(iii) $\Rightarrow$ (ii): Since $\gdim_R k<\infty$, $R$ must be Gorenstein.
Note that every maximal Cohen-Macaulay $R$-module has finite AB-dimension.
Namely, let $X$ be a maximal Cohen-Macaulay $R$-module and let $x_1,x_2,\ldots ,x_t \in \m$ be a maximal $X$-regular sequence.
Since $\abdim_RX/(x_1,x_2,\ldots ,x_t)X<\infty$ by the assumption, we have $\abdim_RX<\infty$ by Proposition \ref{prop}.(2).
Therefore every maximal Cohen-Macaulay $R$-module satisfies (AC).

Let $M$ be an $R$-module which is not maximal Cohen-Macaulay, 
and let $0 \to Y_M \to X_M \to M \to 0$ be a Cohen-Macaulay approximation of $M$.
Since $\abdim_R X_M<\infty$, we have $\abdim_R M<\infty$ by Lemma \ref{ABfin}.
\qed
\end{pfp}

The following is an example which does not satisfy the equality of Theorem \ref{abthm}.(2).(ii).
This example is given by Jorgensen and \c{S}ega \cite{JS}.

\begin{ex}
Let $k$ be a field and let $\alpha \in k$ be a nonzero element.

Put $R=k[x_1,x_2,x_3,x_4]/I$, where $I$ is a ideal of $R$ generated by 

\begin{center}
$\alpha x_1x_3+x_2x_3, x_1x_4+x_2x_4, x_3x_4, x_1^2, x_2^2, x_3^2, x_4^2$
\end{center}

Consider the sequence
$$
{\bf C}:\cdots \overset{d_{i+1}}{\longrightarrow} R^2 \overset{d_{i}}{\longrightarrow} R^2 \overset{d_{i-1}}{\longrightarrow} R^2 \overset{d_{i-2}}{\longrightarrow} \cdots , 
$$
where $d_i$ be a matrix $\left(\begin{smallmatrix}
x_1 & \alpha^i x_3 \\
x_4 & x_2
\end{smallmatrix}\right)$ over $R$.
We put $M=\Coker d_1$.

Jorgensen and \c{S}ega \cite[Lemma 2.2]{JS} prove that ${\bf C}$ is a complete resolution of $M$ and therefore $M$ is totally reflexive.
Furthermore, if $\alpha$ is not a root of unity, then they \cite[Proposition 3.2(b)]{JS} give an $R$-module $T_q$ for every positive integer $q$ such that $\Ext^i_R(M,T_q)\ne 0$ if and only if $i=0,q-1,q$.
This yields that $M$ does not satisfy (AC).
In particular, we have $\abdim_R M=\infty > 0=\gdim_R M$.

On the other hand, we assume that $\alpha$ is a primitive $n$th root of unity.
In this case, $M$ is a periodic module of period $n$ (i.e. $\Omega^nM\cong M$).
For any $N \in M^{\perp}$ and any positive integer $i$, $\Ext^i_R(M,N)\cong\Ext^i_R(\Omega^nM,N)\cong\Ext^{i+n}_R(M,N)\cong\Ext^{i+2n}_R(M,N)\cong \ldots \cong \Ext^{i+jn}_R(M,N)=0$ for $i+jn>\p_R(M,N)$.
Thus we have $\p_R(M,N)=0$ and $\abdim_RM=0$.
\end{ex}

By Theorem \ref{ARC}, we can see that the Auslander condition (AC) is closely related to the Auslander-Reiten conjecture.
Now, we shall show Theorem \ref{ARC}.

\begin{pfa}
Assume that $M$ is not free.
Since there exists a nonsplit exact sequence $0 \to \Omega M \to F_0 \to M \to 0$.
In particular, we see $\Ext^1_R(M,\Omega M) \not= 0$.
On the other hand, we have $\p_R(M,M)=\p_R(M,R)=0$ by assumption and Lemma \ref{keylem}.
Therefore we have $\p_R(M,\Omega M)=0$ by Lemma \ref{keylem}.
In particular, $\Ext^1_R(M,\Omega M) = 0$.
This is contradiction.
Therefore $M$ is free.
\qed
\end{pfa}

In the end of this paper, we give a corollary of this theorem.

\begin{cor}
Let $M$ be an $R$-module which has finite AB-dimension.
If $\Ext^{\gg 0}_R(M,M)=0$,
then $\pd _RM<\infty$.
\end{cor}

\begin{pf}
We put $t=\depth R-\depth M$ and $M'=\Omega^tM$.
Then we see $\abdim_RM'=0$ by Proposition \ref{prop}.(1) and Theorem \ref{abthm}.(2).(i).
Since we can check that $\Ext^{\gg 0}_R(M',M')=0$,
we have  $\Ext^{> 0}_R(M',M')=0$ by Lemma \ref{keylem}.
Therefore $M'$ is free by Theorem \ref{ARC} and we have $\pd _RM=t<\infty$.
\qed
\end{pf}


%

\end{document}